%% file: Shear_Spiralling.tex
\documentclass[a4paper]{amsart}
\usepackage[dvipsnames]{xcolor}
\usepackage{amsmath,amssymb,amsthm,mathrsfs,mathtools}
\usepackage{graphicx}
\usepackage[dvipsnames]{xcolor}
\usepackage{enumitem}
\usepackage[colorlinks]{hyperref}
\usepackage{hyperref}
 \hypersetup{
     citecolor=Periwinkle,
     colorlinks=true,
     linkcolor=Periwinkle,
     filecolor=Periwinkle,      
     urlcolor=black,
     }
\usepackage{fancyhdr} 
\usepackage{float}
\usepackage{subfigure}
\usepackage{wrapfig}
\usepackage{array}
\usepackage{prettyref}
\usepackage{overpic}
\usepackage{orcidlink}
\usepackage{mwe}
\usepackage{faktor}
\usepackage{cite}

\include{commands.tex}
\title[Bounding shears of spiralling triangulations]{Bounding shears of spiralling triangulations on hyperbolic surfaces}
\author{Marie Abadie \hspace{0.1 cm}\orcidlink{0000-0002-8951-4643}}
\email{marie.abadie@uni.lu}
\address{Department of Mathematics, University of Luxembourg, 6, Av. de la Fonte, 4364 Esch-sur-Alzette,
Luxembourg}
\date{\today}

\DeclareRobustCommand{\SkipTocEntry}[5]{}
\begin{document}
\blfootnote{MSC classes: 30F60, 57M50}

\begin{abstract}
    We show that any hyperbolic surface admits a spiralling triangulation with bounded shear parameters along its bi-infinite components. This upper bound depends logarithmically on the topology of the surface.
\end{abstract}

\maketitle
\tableofcontents

\section{Introduction}
A theorem of Bers states that for any hyperbolic surface, there exists a constant, depending only on the topology of the surface, that bounds from above the lengths of curves in its shortest pants decomposition~\cite{Bers1974,Bers1985}. Much work has been put into quantifiying these constants, not only for closed hyperbolic surfaces but also for more general cases involving surfaces with punctures or geodesic boundary components\cite{Buser, Buser1992,BalacheffParlier2012,BalacheffParlier2012ShortLD}. More precisely, 
\begin{BersThm*}
    Let $g,n\geq 0$ with $2g-2+n>0$. Then there exists a constant $B=B(g,n)$ such that any hyperbolic surface $X$ of genus $g$ with $n$ punctures admits a pants decomposition $P=(\gamma_1,\dots, \gamma_{3g-3+n})$ satisfying
    \[
        \ell_X(\gamma_i) \leq B,
    \]
    for all $i=1,\dots,3g-3+n$.
\end{BersThm*}
%Equivalenty, \[\min_{\mathcal{P}} \max_{\gamma \in \mathcal{P}}\ell_X(\gamma) \leq B_{g,n},\]where $\mathcal{P}$ ranges over all pants decompositions of $X$. 
The Bers constant is not known explicitly, the best known upper bound grows linearly with the genus and the number of punctures~\cite{Parlier2023}. Instead of pants decompositions, one can decompose hyperbolic surfaces into ideal triangles and replace the length parameters by shear parameters, encoding how the ideal triangles are glued together~\cite{Bonahon1996,fock1998, thurston1998}. In the present work, we show that, similarly to the existence of pants decompositions with uniformly (depending on the topology) bounded length, all hyperbolic surfaces admit spiralling triangulations with uniformly bounded shear parameters.
%Given a hyperbolic surface $X$, we thus aim at finding a ``nice'' spiralling triangulation $\mathcal{T}$ of $X$, that is, whose shear parameters $S_{\lambda}^{\mathcal{T}}(X)$, where $\lambda$ is a the geodesic arc components of $\mathcal{T}$, are all small. In other words,
\begin{theorem*}[\ref{thm:main}]
    Let $X$ be an  hyperbolic surface of genus $g$ with $n$ punctures where $2g-2+n>0$. Then $X$ admits a spiralling triangulation $\mathcal{T}$ satisfying 
    \[
        |S_{\lambda}^{\mathcal{T}}(X)| \leq 12\log(8\pi(2g-2+n))+25,
    \]
    for any arc component $\lambda \in \mathcal{T}$.
\end{theorem*}

Jiang studied lower and upper bounds for the ``minimal shearing'' of punctured hyperbolic surfaces along ideal triangulations whose arcs all end in cusps~\cite{ManmanJiang2021}. In the present work, surfaces are not necessarily punctured and we consider triangulations where arcs either end in cusps or spiral around simple closed geodesics. 

The strategy of the proof goes as follows. First, given a hyperbolic surface $X$ we consider a ``short'' hexagon decomposition $(\Gamma, \mathcal{A})$ on $X$, where $\Gamma$ is a disjoint collection of simple closed geodesics and $\mathcal{A}$ is a maximal collection of orthogeodesic arcs that cut the surface into a disjoint union of right-angled hexagons. An arc $a\in \mathcal{A}$ has $0$, $1$ or $2$ endpoints on components of $\Gamma$. The existence of a ``short'' hexagon decomposition
follows by extending~\cite[Theorem 1.3]{Parlier2016}, originally stated for compact surfaces, to the case of punctured surfaces, see Proposition~\ref{Hexagon-Decomposition-Existence}. From this ``short'' hexagon decomposition $(\Gamma, \mathcal{A})$, we construct an ideal triangulation $\mathcal{T}_{(\Gamma, \mathcal{A})}$ as follows. Let $\lambda_a$ be the geodesic arc obtained from $a\in \mathcal{A}$ by making it spiral infinitely many times around the simple closed geodesic at each of its endpoint. This gives an ideal triangulation \[\mathcal{T}_{(\Gamma, \mathcal{A})}\coloneq \left(\lambda_a\right)_{a \in \mathcal{A}}.\] By an area argument, the spiralling part of the arc $\lambda_a$ does not contribute to the shear parameter along that arc. Thus, the shear can be bounded by the length of an appropriate truncation $\lambda_a^t$ of $\lambda_a$ and we obtain
\[ 
    |S_{\lambda_a}^{\mathcal{T}_{(\Gamma, \mathcal{A})}}(X)| \leq \ell_X(\lambda_a^t).
\]
Since the original orthogeodesic arcs $a$ are ``short'', we bound the length of the truncated subarcs $\lambda_a$ by the lengths of the hexagon sides.

\addtocontents{toc}{\SkipTocEntry}
\subsection*{Acknowledgements} 
I wish to thank Hugo Parlier for his support, useful discussions and comments. I would also like to thank Wai Yeung Lam, Didac Martinez-Granado and Katie Vokes for helpful comments that improved the readability. Additionally, I am thankful to the anonymous referee for their careful reading and constructive corrections. This work was partially supported by the Luxembourg National Research Fund OPEN grant O19/13865598.

\section{Background}

\subsection{Teichmüller space and Fenchel-Nielsen coordinates} Let $\Hh^2$ denote the hyperbolic plane and fix $\Sigma_{g,n}$ an oriented surface of genus $g \geq 1$ with $n\geq 0$ cusps, such that $2g - 2 + n > 0$. A \emph{marked hyperbolic surface} $(g,n)$ is a pair $(X, \phi)$ where $X$ is an oriented hyperbolic surface and $\phi: \Sigma_{g,n} \longrightarrow X$ is an orientation-preserving homeomorphism. Two marked hyperbolic surfaces $(X,\phi)$ and $(X', \phi')$ are equivalent if there exists an isometry $\iota:X\to X'$ such that $\phi'$ is isotopic to $\iota \circ \phi$. The \emph{Teichmüller space} of $\Sigma_{g,n}$, denoted by $\teich(\Sigma_{g,n})$ or $\teich_{g,n}$, is the space of equivalence classes of marked hyperbolic surfaces of signature $(g,n)$. 

For any essential closed curve $\gamma \subset \Sigma_{g,n}$ we define the length function of $\gamma$,
\begin{center}
    $\begin{array}[t]{lrcl}
& \ell_{\cdot}(\gamma): \teich_{g,n} & \longrightarrow &   \R_{\geq 0} \\  
  & (X,\phi) & \longmapsto & \ell_X(\gamma)\\
\end{array} $
\end{center}
which assigns to each marked surface $(X, \phi)$ the length of the unique geodesic representative in the homotopy class of $\phi(\gamma)$ on $X$.

The Teichmüller space is homeomorphic to $\R^{6g-6+2n}$ via the \emph{Fenchel-Nielsen coordinates}. These coordinates depend on a choice of \emph{pants decomposition} $\mathcal{P}=(\gamma_1,\ldots, \gamma_{3g-3+n})$ of $\Sigma_{g,n}$, that is a maximal collection of disjoint simple closed curves on $\Sigma_{g,n}$. The complement $\Sigma_{g,n}\smallsetminus \mathcal{P}$ is a disjoint union of $(2g-2+n)$ three-holed spheres, called \emph{pair of pants}. The Fenchel-Nielsen coordinates associate to a point of Teichmüller space the length and twist parameters along each component of the pants decomposition $\mathcal{P}$.

\subsection{Ideal and Spiralling triangulations, Hexagon decompositions}

A \emph{spiralling triangulation} $\mathcal{T}$ of $\Sigma_{g,n}$ is a finite maximal geodesic lamination, i.e., a closed finite union of disjoint simple (complete) geodesics, whose components are either closed curves or bi-infinite arcs ending in cusps or spiralling around closed components. By maximality, the complement of $\mathcal{T}$ consists of $(4g-4+2n)$ ideal triangles, whose edges are the $6g-6+3n$ bi-infinite components of $\mathcal{T}$. There may also be $0$ to $(3g-3+n)$ closed components, determined by the spiralling bi-infinite arcs. In the case where no spiralling occurs, i.e. all bi-infinite arcs end is cusps, $\mathcal{T}$ is called an \emph{ideal triangulation}. Abusing notation, $\mathcal{T}$ also refers to the set of bi-infinite components of the ideal triangulation.

A spiralling triangulation is associated with a \emph{hexagon decomposition} of $\Sigma_{g,n}$ which consists of a collection $\Gamma$ of disjoint simple closed geodesics together with a maximal collection $\mathcal{A}$ of disjoint orthogeodesic arcs so that $\Sigma_{g,n}\smallsetminus (\Gamma, \mathcal{A})$ is a disjoint union of generalized right-angled hexagons. The maximal collection $\mathcal{A}$ contains $6g-6+3n$ orthogeodesic arcs, and the surface is decomposed into $4g-4+2n$ geodesic generalized right–angled hexagons. A generalized right-angled hexagon is one of the following: a right-angled hexagon, a right-angled pentagon with one ideal point, a quadrilateral with two adjacent ideal points and two right angles, or an ideal triangle. We will occasionally abuse notation by referring to a generalized right-angled hexagon simply as a right-angled hexagon.

\vspace{-0.2cm}
\begin{figure}[h]
    \centering
    \includegraphics[scale=0.1]{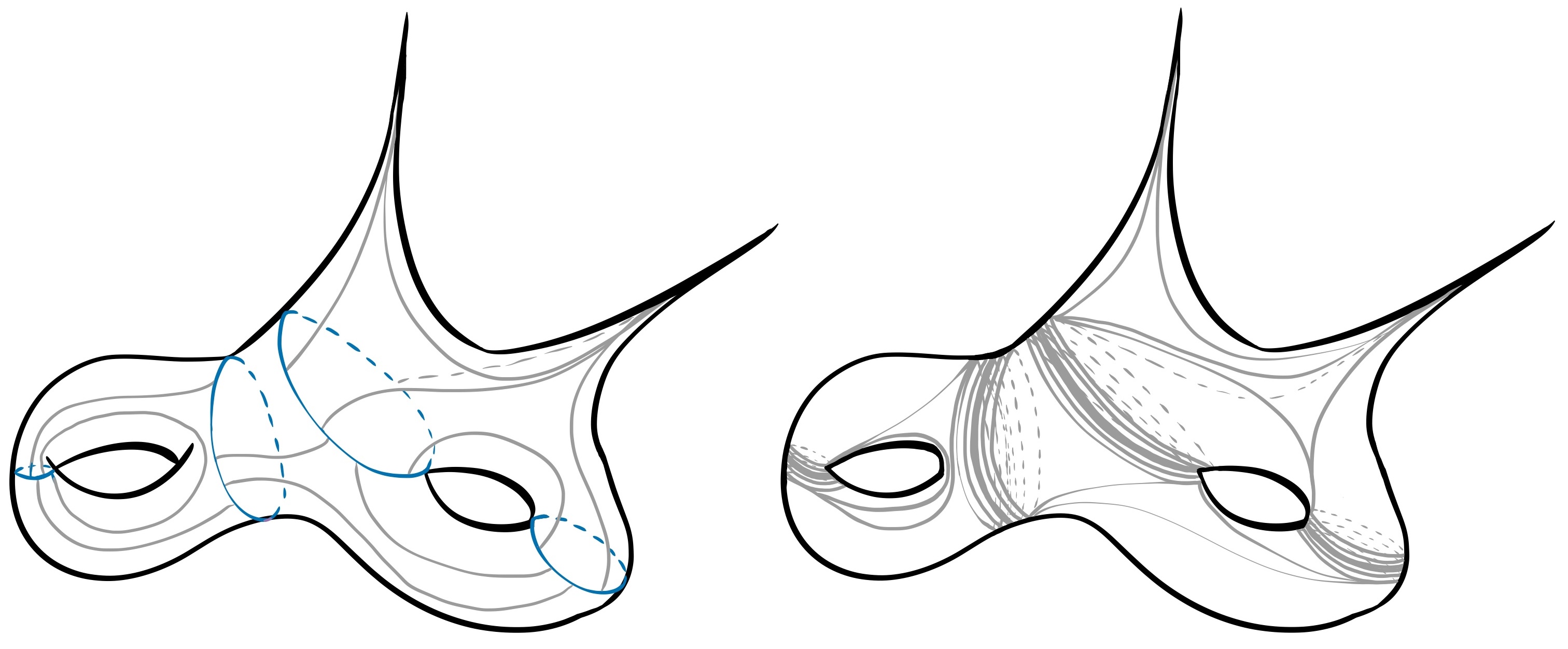}
    \caption{A hexagon decomposition and its associated spiralling triangulation.}
    \label{Spiralling_tirangulation}
\end{figure} 

From a hexagon decomposition $(\Gamma, \mathcal{A})$ we construct a spiralling triangulation $T_{(\Gamma, \mathcal{A})}$. Define a family of infinite geodesic arcs $(\lambda_a)_{a \in \mathcal{A}}$ as follows:
\begin{itemize} 
    \item[-] If $a$ has both ends in a cusp, we define $\lambda_a$ to be equal to $a$.
    \item[-] If $a$ has endpoints on $\gamma_1$ and $\gamma_2$ in $\Gamma$, let $\lambda_a$ be the geodesic arc homotopic to the arc which follows $a$ and spiral infinitely many times around $\gamma_1$ and $\gamma_2$. It is possible that $\gamma_1=\gamma_2$.
    \item[-] If $a$ has one endpoint in a cusp and the other on a simple closed curve $\gamma \in \Gamma$, then let $\lambda_a$ be the geodesic arc homotopic to the arc which follows $a$ from the cusp to $\gamma$ and spiralling infinitely many times around $\gamma$.
\end{itemize} 
There are different possible choices of spiralling. Pick one, once and for all: for each component of $\Gamma$, we choose an orientation and define the spiralling direction accordingly.
Then, we define the \emph{spiralling triangulation} $T_{(\Gamma, \mathcal{A})}$ associated with $(\Gamma, \mathcal{A})$ to be the collection of orthogeodesic arcs $(\lambda_a)_{a \in \mathcal{A}}$ together with the curves in $\Gamma$. See Figure~\ref{Spiralling_tirangulation}.

\subsection{Shearing coordinates}
Instead of decomposing a surface into pair of pants one can decompose it into ideal triangles. To such a decomposition one can associate \emph{shear parameters}, which we describe below. 

Any two hyperbolic ideal triangles  are isometric to each other, in particular they are isometric to the \emph{standard} triangle $(0,1,\infty)$ in the upper half-plane model. Every ideal triangle has an inner circle tangent to its sides, called its \emph{inscribed circle}. it has radius $2\rho$, where \[\rho\coloneq\frac{\log(3)}{4}.\]

\begin{figure}
    \centering
    \begin{overpic}[width=.6\linewidth,keepaspectratio]{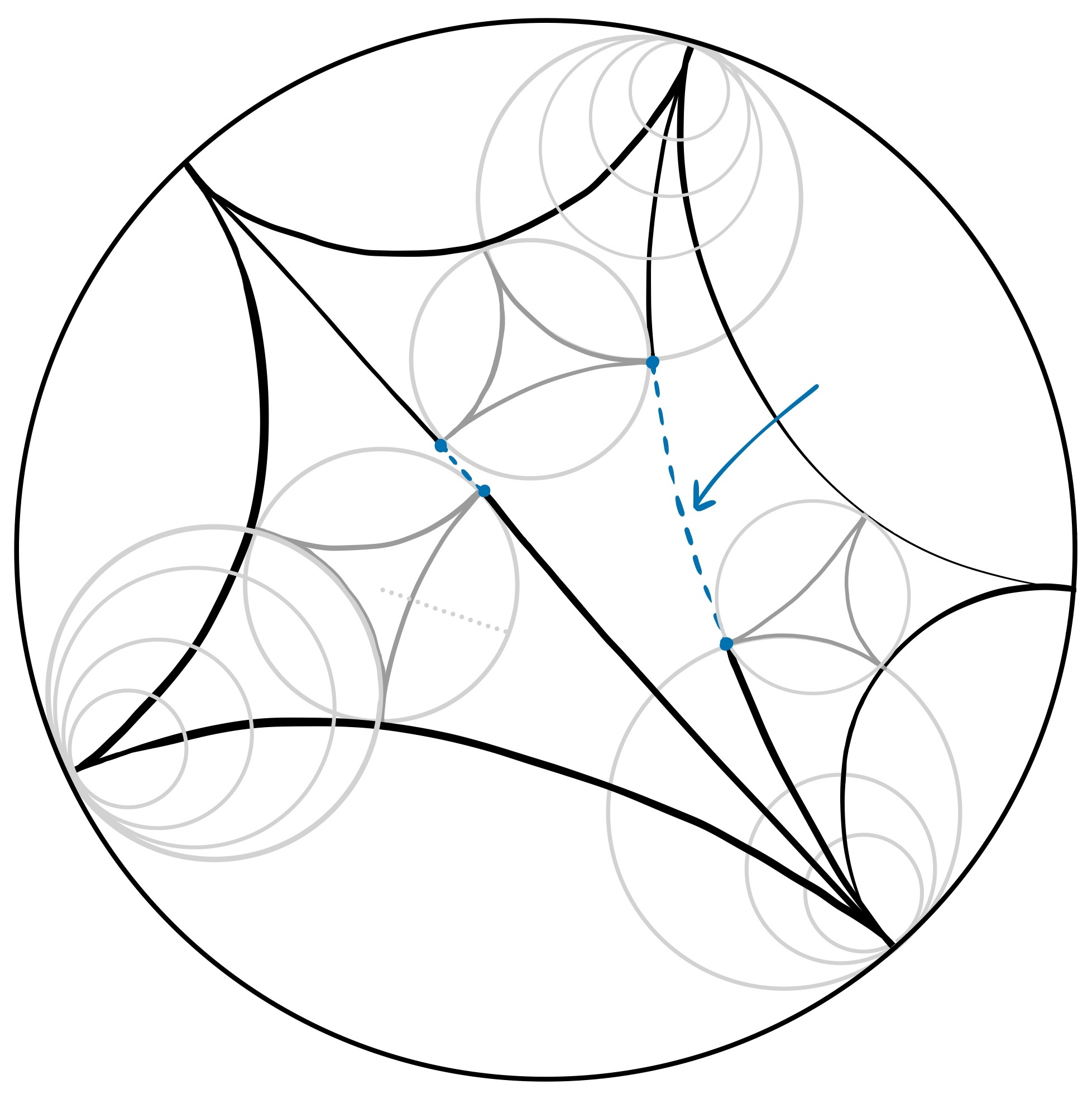}
    \put(27,79){\Large$\Delta$}
    \put(65.75,73){\Large$\Delta'$}
    \put(55.5,78){$\lambda$}
    \put(41,44.75){\small\color{Gray}$2\rho$}
    \put(73,66){\color{RoyalBlue}$S_{\lambda}(\Delta, \Delta')$\color{black}}
\end{overpic}\vspace*{-0.25cm}
    \caption{Shear along the common edge of two ideal triangles.}
    \label{shear2}
\end{figure}

The tangency points are called \emph{shear points}. Two ideal triangles $\Delta$ and $\Delta'$ that share an edge $\lambda$ determine two shear points on $\lambda$. The signed distance (with respect to a fixed choice of orientation of $\lambda$) between those shear points is called the \emph{shear} of $\Delta$ and $\Delta'$ along $\lambda$ and is denoted by $S_{\lambda}(\Delta,\Delta')$.

The inscribed circle of an ideal triangle $\Delta$ contains an equilateral contact triangle whose vertices are the three shear points of $\Delta$ and whose edges are horocyclic segments perpendicular to the sides of the ideal triangle. The contact triangle can also be obtained by foliating $\Delta$ by horocycles around each vertex: it is formed by the three foliating leaves that meet tangentially.

Now we define the shear parameters of $X$ with respect to a spiralling triangulation $\mathcal{T}$, with bi-infinite components $\lambda_1,\dots,\lambda_{6g-6+3n}$,
\begin{center}
    $\begin{array}[t]{lrcl}
    & s_{T}:\teich_{g,n} & \longrightarrow &   \R^{6g-6+3n}  \\  
    & (X,\phi) & \longmapsto & (S^{\mathcal{T}}_{\lambda_1}(X), \ldots, S^{\mathcal{T}}_{\lambda_{6g-6+3n}}(X))\\
\end{array} $
\end{center}
Fix a lift $\widetilde{\mathcal{T}}$ of $\mathcal{T}$ to the hyperbolic plane $\Hh^2$. Its complement $\Hh^2-\widetilde{\mathcal{T}}$ is a disjoint union of ideal triangles. Let $1\leq i\leq 6g-6+3n$ and pick a lift $\widetilde{\lambda}_i\subset\widetilde{\mathcal{T}}$ of $\lambda_i$. It is the common edge of two distinct ideal triangles $\widetilde{\Delta}$ and $\widetilde{\Delta}'$. Then $S^{\mathcal{T}}_{\lambda_i}$ is defined as the shear parameter of $\widetilde{\Delta}$ and $\widetilde{\Delta}'$ along $\widetilde{\lambda}_i$, i.e. \[S^{\mathcal{T}}_{\lambda_i}(X) \coloneq S_{\widetilde{\lambda}_i}(\widetilde{\Delta},\widetilde{\Delta}').\]
The shear parameter of $X$ along an edge $\lambda$ in $\mathcal{T}$, written as $S^{\mathcal{T}}_{\lambda}(X)$, is also denoted $S_{\lambda}^{\mathcal{T}}$ when there is no ambiguity.

\begin{remark}
    The shear parameters depend up to sign to choices of orientation for the $\lambda_i$. We assume them to be fixed once and for all.
    
    The shear parameters are not all independent: the sum of shears along all arcs coming from the same cusp is zero, and the sum of the shears of all arcs spiralling around the same closed curve (on the same side) equals the length of that curve\cite{chekhov2004,Bestvina2009}. If all the triangles have their vertices in cusps, there are then $6g-6+2n$ independent shearing coordinates which provide a global parameterization of Teichmüller space. However, when $\mathcal{T}$ has closed components, the Fenchel-Nielsen twists around them need to be added to provide a global parameterization, called the \emph{shearing coordinates}~\cite{Bonahon1996,fock1998, thurston1998}.
\end{remark}

\subsection{Collars and cusp neighborhoods}
We now recall the definitions and basic properties of collars and cusps neighborhoods. Let $\wstd(\ell)$ be defined for $\ell > 0$ as
\[
    \wstd(\ell) =  \arcsinh\left(\frac{1}{\sinh\left( \frac{\ell}{2}\right)}\right).
\]
Let $\gamma$ be a simple closed geodesic on $X$, its \emph{standard collar} is
\[
    \mathcal{C}_{\wstd(\ell_X(\gamma))}(\gamma) = \{\, p \in X \mid d(p,\gamma) \le \wstd(\ell_X(\gamma)) \}.
\]
When there is no ambiguity, to simplify notation, we write $\mathcal{C}_{\wstd}(\gamma)$ and $\wstd(\gamma)$. For $\delta \leq \wstd(\gamma)$, the subset of points at distance at most $\delta$ from $\gamma$ is the \emph{$\delta$-collar} and is denoted by $\mathcal{C}_\delta(\gamma)$. We say that the collar $\mathcal{C}_\delta(\gamma)$ has width $\delta$.

If $X$ has $n>0$ punctures then each puncture corresponds to a \emph{cusp}. A cusp is a region isometric to the cylinder
\[ (-\infty, \log 2] \times \mathbb{S}^1 \]
equipped with the Riemannian metric
\[ d\rho^2 + e^{2\rho} \, dt^2. \]
Each cusp is bounded by a horocycle of length $2$~\cite{Keen1974}. Given a cusp $c$ on $X$, the subset of the cusp bounded by a horocycle of length $\delta \le 2$ is called a $\delta$-\emph{cusp neighborhood}, denoted by $N_\delta(c)$. When $\delta=2$, it is called the \emph{standard cusp neighborhood}. The following result will be useful in later arguments; 

\begin{lemma}[Collar Theorem~\cite{Keen1974}, Theorem 4.4.6~\cite{Buser}]\label{CollarThm}
Let $X$ be a hyperbolic surface. If $\gamma_1$ and $\gamma_2$ are two disjoint simple closed geodesics on $X$, then their standard collars $\mathcal{C}_{w(\gamma_1)}$ and $\mathcal{C}_{w(\gamma_2)}$ are disjoint embedded cylinders. Moreover, if $X$ has punctures, the standard collars around disjoint simple closed geodesics and the standard cusp neighborhoods are pairwise disjoint.
\end{lemma}

\subsection{Truncation and short hexagon decompositions}
By Theorem 1.3 from Parlier~\cite{Parlier2016}, any compact hyperbolic surface admits a ``short'' hexagon decomposition. In this subsection we generalize the original statement and extend it to punctured hyperbolic surfaces. 

\begin{definition}\label{def:truncation_function}
    A \emph{truncation function} is a continuous function $ \wabstract \colon \R_{> 0}\to\R_{\geq 0}$ satisfying the following properties:
    \begin{enumerate}
        \item For all $\ell\in\R_{\geq 0}$, we have $0\leq \wabstract(\ell)\leq \wstd(\ell)$, where $\wstd(\ell)$ is the width of the standard collar around a simple closed geodesic of length $\ell$.
        \item The limit
        \begin{equation*}
            h_0(\wabstract)\coloneqq\lim\limits_{\ell\to 0}\ell\cosh\left(\wabstract(\ell)\right)
        \end{equation*}
        is strictly positive.
        \item The difference between $\wstd$ and $\wabstract$ is uniformly bounded from above; that is, there exists a universal constant $\Delta_{\wabstract} > 0$ such that for all $\ell \geq 0$,
        \begin{equation*}
            \wstd(\ell) - \wabstract(\ell) \leq \Delta_{\wabstract}.
        \end{equation*}
    \end{enumerate}
\end{definition}

\begin{definition}
    Let $\Gamma\subset\Sigma$ be a simple multicurve and $\wabstract$ a truncation function.
    \begin{enumerate}
        \item For $X\in\teich_{g,n}$, its $(\Gamma,\wabstract)$\emph{-truncation} is the subsurface
        \begin{equation*}
            X^{\wabstract}_{\Gamma}\coloneqq X \smallsetminus\left(\bigcup_{i=1}^n N_{h_0(\wabstract)}(c_i) \cup \bigcup_{\gamma\in\Gamma} \mathcal{C}_{\wabstract(\ell_X(\gamma))}(\gamma)\right).
        \end{equation*}
        \item The $(\Gamma,\wabstract)$\emph{-truncation} of an orthogeodesic arc $a$ on $X-\Gamma$ is the subarc
        \begin{equation*}
            a^{\wabstract}_{\Gamma}\coloneqq a\cap X^{\wabstract}_{\Gamma}
        \end{equation*}
        and its $(\Gamma,\wabstract)$\emph{-truncation length} is defined as
        \begin{equation*}
            \ell^{\wabstract}_{X,\Gamma}(a)\coloneqq\ell_X(a^{\wabstract}_{\Gamma}).
        \end{equation*}
    \end{enumerate}
\end{definition}

To simplify notation, we write $\mathcal{C}_{\wabstract}(\gamma)$ instead of $\mathcal{C}_{\wabstract(\ell_X(\gamma))}(\gamma)$ when there is no ambiguity.

\begin{definition}\label{ShortHexagonDecomposition}
    Let $\wabstract:\R_{> 0}\rightarrow \R_{>0}$ be a truncation function and let $L, L_A>0$. A hexagon decomposition $(\Gamma,\mathcal{A})$ on $X \in \overline{\teich}_{g,n}$ is $(L,L_A, \wabstract)$-\emph{short} if:
    \begin{enumerate}
        \item for every $\gamma\in\Gamma$,
        \[
        \ell_X(\gamma)\le L,
        \]
        \item for every $a\in\mathcal{A}$,
        \[
            \ell_{X, \Gamma}^{\wabstract}(a)\le L_A.
        \]
    \end{enumerate}
\end{definition}

If $L=L_A$ we say that $(\Gamma,\mathcal{A})$ is $(L,\wabstract)$-\emph{short}.

\begin{lemma}\label{lem:Bavard}
    For any $X \in \teich_{g,n}$ and any $x \in X$, there exists a geodesic loop $\delta_x$ based at $x$ such that
    \[ 
    \ell_X(\delta_x) \le 2 R_{g,n},\] where \[R_{g,n} := \arcch\Bigg(\frac{1}{2 \, \sin\big(\frac{\pi}{12g - 6 + 6n}\big)}\Bigg).
    \]
\end{lemma}
\begin{proof}
    The injectivity radius of $X$ at a point $x$ is the radius of the largest open metric disk embedded at $B$; it is denoted by $r_x(X)$. By~\cite[Lemma~4.1.5]{Buser} $\ell_X(\delta_x)=2R_x(X)$. The statement was proved by Bavard in~\cite[Theorem 1.2)]{Bavard1996} for the compact case. For the punctured case the proof is the same, we just need to update the formula for the area.
\end{proof}

\begin{lemma}\label{lem:boundlength2}
 Let $\wabstract:\R_{> 0}\rightarrow \R_{\geq 0}$ be a truncation function. Let $\delta_x \subset X$ be a simple geodesic loop based at a point $x\in X$, and let $\delta$ be the unique simple closed geodesic freely homotopic to $\delta_x$. Then
    \[  
    \sup_{p \in c} d_X\left(x, \mathcal{C}_{\wabstract(\ell_X(\delta))}(\delta)\right) < \log \left( \sinh\left( \frac{\ell_X(\delta_x)}{2} \right)\right) + \Delta_{\wabstract}.
    \]
\end{lemma}

\begin{proof}
    The case $\wabstract=\wstd$ is proved by Parlier in~\cite[Lemma 3.8]{Parlier2016}. The general statement follows by updating the width of the collar in the proof.
\end{proof}
 
 %Let $L_1\coloneq 2\log(8\pi(2g-2+n))$ and note that $2R_{g,n} \leq L_1$.
 
\begin{proposition}\label{Hexagon-Decomposition-Existence}
    Let $\wabstract:\R_{> 0}\rightarrow \R_{\geq 0}$ be a truncation function and fix 
    \[L= 2R_{g,n} \quad \text{and} \quad L_A= R_{g,n} +  2\Delta_{\wabstract}.\] 
    Then, for any $X \in \teich_{g,n}$ with $2g-2+n >0$, there exists a $(L,L_A,\wabstract)$-short hexagon decomposition $(\Gamma, \As)$ on $X$.
\end{proposition}

\begin{proof}
    Let us start with a subsurface $X_{0} \subset X$, defined as follows. If $n\geq 1$, we remove the cusp neighborhoods and set
    \[
        X_0\coloneqq X \smallsetminus \bigcup_{i=1}^n N_{h_0(\wabstract)}(c_i).
    \]
    If $n=0$, fix an arbitrary point $x\in X$. By Lemma~\ref{lem:Bavard}, there exists a simple closed curve $\alpha_x$ based at $x$ such that $\ell_X(\alpha_x) \leq 2R_{g,n}$. Let $\alpha$ be the unique simple closed geodesic freely homotopic to $\alpha_x$ and set 
    \[
        X_0\coloneqq X \smallsetminus \mathcal{C}_{\wabstract}(\alpha).
    \]
    We construct a nested sequence of subsurfaces $X_0,\dots,X_m$ by iteratively cutting out $w$-collars around disjoint ``short'' simple closed geodesics until the following condition is satisfied:
    \begin{center}
        ($\ast$) all points lie at distance at most $D\coloneq\frac{R_{g,n}}{2} + \Delta_{\wabstract}$ from the boundary.
    \end{center}
    Assume that $X_k$ has been constructed, for some $k\geq 0$ and that ($\ast$) is not satisfied. Then there is a point $x$ such that 
    \begin{equation*}
        d_{X_k}(x, \partial X_k) > D.
    \end{equation*}
    By Lemma~\ref{lem:Bavard}, there is a non-trivial simple closed curve $\delta_x$ based at $x$, satisfying
    \begin{equation*}
       \ell_{X_k}(\delta_x) \leq 2R_{g,n}. 
    \end{equation*}
    Let $\delta$ be the unique simple closed geodesic freely homotopic to $\delta_x$. By Lemma~\ref{lem:boundlength2}, the distance from $x$ to the $\wabstract$-collar of $\delta$ is bounded by
    \begin{align*}
         d_{X_k}(x, \mathcal{C}_{\wabstract}(\delta)) &\leq \log \left( \sinh\left( \frac{\ell_X(\delta)}{2} \right)\right) + \Delta_{\wabstract} \notag\\
         &\leq \log \left( \sinh\left( R_{g,n} \right)\right) + \Delta_{\wabstract} \notag\\
         &\leq \frac{R_{g,n}}{2} + \Delta_{\wabstract} =D.
    \end{align*}
    Because $d_{X_k}(x, \partial X_k) > D$ while $d_{X_k}(x, \mathcal{C}_{\wabstract}(\delta))\leq D $, we obtain $d_{X_k}(\delta, \partial X_k) > 0$. Hence, $\delta$ is contained in $X_k$ and is not a boundary component of $X_k$. Define
    \[
        X_{k+1}=X_k \smallsetminus \mathcal{C}_{\wabstract}(\delta).
    \]
    and repeat the process. This process must end after a finite number of steps since a surface of genus $g$ with $n$ punctures support at most $3g-3+n$ disjoint simple closed geodesics. We thus obtain a subsurface $X_m\subset X$ satisfying ($\ast$).

    Let $\Gamma$ be the set of all simple closed geodesics cut during this process. By construction, for all $\gamma \in \Gamma$
    \begin{equation*}
       \ell_{X}(\gamma) \leq 2R_{g,n}
    \end{equation*}
    Let $X_m$ be the subsurface obtained at the end of the process, that is
     \begin{equation*}
            X_m = X \smallsetminus\left(\bigcup_{i=1}^n N_{h_0(\wabstract)}(c_i) \cup \bigcup_{\gamma\in\Gamma} \mathcal{C}_{\wabstract(\ell_X(\gamma))}(\gamma)\right).
        \end{equation*}
    By construction, there are no points left in $X_m$ whose distance to the boundary exceeds the threshold. Therefore, for all $x \in X_m$,
    \[
        d_{X_m}(x, \partial X_m) \leq D.
    \]

    \begin{figure}
        \centering
        \begin{overpic}[width=.6\linewidth,keepaspectratio]{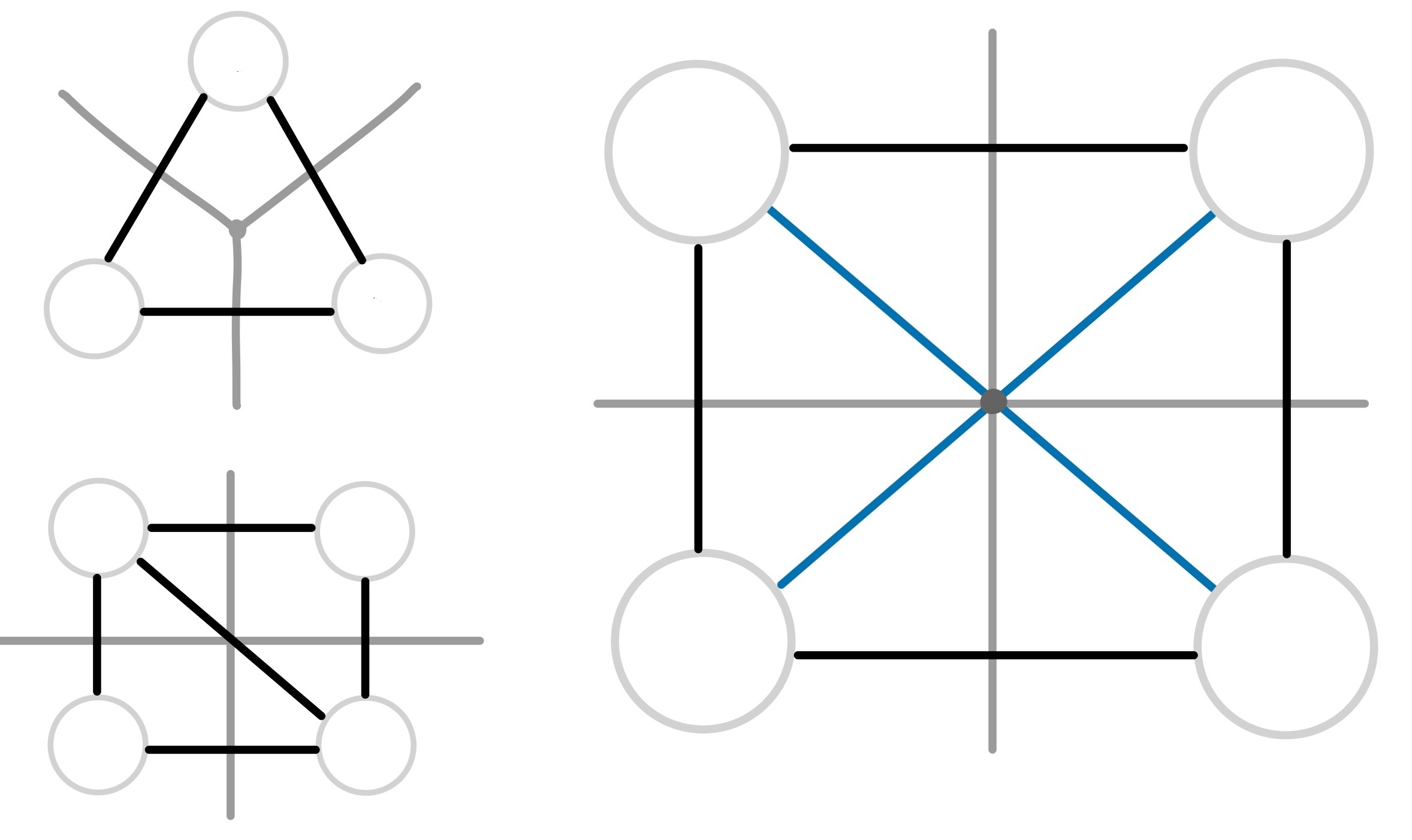}
        \put(33,33){\color{gray}$\spine(X_m)$}
        \put(93,32){\text{dual graph}\color{black}}
        \put(59,35){\color{RoyalBlue}$a_1$}
        \put(63,22){\color{RoyalBlue}$a_4$}
        \put(80.5,23){\color{RoyalBlue}$a_3$}
        \put(80,36){\color{RoyalBlue}$a_2$}
        \put(70.5,33){\color{gray}$v$}
    \end{overpic}\vspace*{-0.25cm}
        \caption{Construction of an ideal triangulation from a Voronoï cell decomposition}\label{Voronoi2}
    \end{figure}
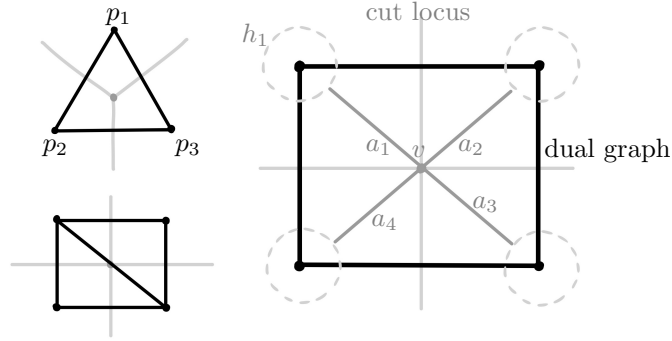

    Let $B_m$ denote the collection of boundary components of $X_m$. We consider the Voronoï cell decomposition of $X_m$ with respect to $B_m$ as defined in~\cite{BowditchEpstein1988},~\cite[Section 2.1]{Mondello2009} and~\cite[Section 2.2]{Budd2025}. We present briefly the construction. The valence of a point $x\in X_m$, denoted by $\nu(x)$ is the number of shortest geodesics from $x$ to $B_m=\partial X_m$ realizing the distance $d(x,B_m)$. Note that $\nu(x) \geq 1$ and two distinct shortest geodesics at one point can end on the same boundary component. Then define \[ V= \left\{x \in X_m \mid \nu(x)\geq 3 \right\} \text{ and } N= \left\{x \in X_m \mid \nu(x)=2 \right\}. \]
    Observe that $V$ is a finite collection of points called \emph{vertices} while $N$ is a finite disjoint union of simple open geodesic arcs called \emph{edges}. Their union $V \cup N$ is a $1$-dimensional CW-complex called the \emph{spine} of $X_m$ and is denoted by $\spine(X_m)$. For $e \in N$, let $\beta_e \in \spine(X_m)$ be an edge of the spine of $X_m$. Pick any point $x \in \beta_e$ and let $\gamma_1$ and $\gamma_2$ be the two shortest geodesics joining $x$ with $B_m$. Let $\alpha_e$ be the orthogeodesic arc in the isotopy class of $\gamma_1 \cup \gamma_2$, that is meeting $B_m$ perpendicularly. We say that $\alpha_e$ is dual to $\beta_e$. Recall that by construction of $X_m$, $\alpha_e$ has length at most $2D$ since $x$ lies at distance at most $D$ from $\partial X_m$.
        
    In our situation the $\spine(X_m)$ is a tree, if it is trivalent, that is if all the vertices have valence 3, then the collection of dual arcs $(\alpha_e)_{e\in N}$ forms a valid hexagon decomposition of $X_m$ and is called the \emph{dual graph} of the spine. If $V$ contains a vertex $v$ of valence $\nu(v)=k > 3$, we add arcs to its dual graph as follows. Let $(a_i)_{i=1}^{k}$ be the family of shortest geodesic from $v$ to $B_m$. Consider the isotopy classes of the $k-3$ concatenations $a_1 \star a_3, \dots, a_1 \star a_{k-1}$. Then we add to the dual graph the orthogeodesic representatives from each isotopy classes. Repeating this process for all vertices of degree greater than $3$ yields an augmented dual graph that defines a valid hexagon decomposition of $X_m$, see Figure~\ref{Voronoi2}. Since each $a_i$ has length at most $D$ on $X_m$, the new arcs have lengths at most $2D$. 

    Because $X_m$ was obtained by removing disjoint collars and cusp neighborhoods, we can continuously extend the arcs of this augmented dual graph from $X_m$ back across the collars ad cusps on the full surface $X$. This construction yields an arc set $\mathcal{A}$ such that $(\Gamma, \mathcal{A})$ forms a hexagon decomposition of $X$. 

    To conclude, by construction, any curve $\gamma\in \Gamma$ satisfies 
    \[\ell_X(\gamma) \leq 2R_{g,n} = L.\] Furthermore, the truncated lengths of the arcs satisfy 
    \[\ell_{X, \Gamma}^{\wabstract}(a) \leq 2D = R_{g,n} + 2\Delta_{\wabstract} = L_A.\] 
    Thus, the decomposition is $(L,L_A, \wabstract)$-short, completing the proof.
\end{proof}

\begin{cor}\label{corExistSdtShort}
    Let $L=R_{g,n} \leq \log(8\pi(2g-2+n))$. Then, any $X\in \teich_{g,n}$ admits a $(2L,L,\wstd)$-short hexagon decomposition $(\Gamma, \mathcal{A})$ on $X$.
\end{cor}

\begin{proposition}\label{ComparingLshort}
    Let $X\in \teich_{g,n}$ and  $\wabstract_1, \wabstract_2:\R_{>0}\rightarrow \R_{\geq 0}$ be two truncation functions. If $(\Gamma, \mathcal{A})$ is a $(L^1,L_A^1,\wabstract_1)$-short hexagon decomposition on $X$ then it also $(L^2,L_A^2,\wabstract_2)$-short on $X$, where 
    \[L^1=L^2, \qquad L_A^2= L_A^1 + 2(\Delta_{\wabstract_1}+\Delta_{\wabstract_2}).\]
\end{proposition}

\begin{proof}
    For $i=1,2$, take $\Delta_i:=\Delta_{\wabstract_i}$ and set $L_A^2:=L_A^1+2(\Delta_1+\Delta_2)$.
    We show that $(\Gamma,\mathcal{A})$ is $(L^2,L_A^2,\wabstract_2)$-short. Since $(\Gamma,\mathcal{A})$ is $(L^1,L_A^1,\wabstract_1)$-short, for every $\gamma\in\Gamma$ we already have
    \[
    \ell_X(\gamma)\le L^1= L^2.
    \]
    So it remains to prove that for every $a\in\mathcal{A}$,
    \[
    \ell^{\wabstract_2}_{X,\Gamma}(a)\le L_A^2.
    \]
    
    Fix such an arc $a$. The two truncations $a^{\wabstract_1}_\Gamma$ and $a^{\wabstract_2}_\Gamma$ differ only at the two ends of $a$, because the underlying arc is the same and only the widths of the removed collar neighborhoods (and cusp neighborhoods) vary. Each endpoint of $a$ lies either on the boundary of a collar around some $\gamma\in\Gamma$, or on the boundary of a cusp neighborhood.
    
    We first bound the change coming from collars. For any $\ell>0$,
    \[
        0\le \wstd(\ell)-\wabstract_i(\ell)\le \Delta_i,
    \]
    hence
    \[
        \wabstract_1(\ell)-\wabstract_2(\ell)
        =\bigl(\wstd(\ell)-\wabstract_2(\ell)\bigr)-\bigl(\wstd(\ell)-\wabstract_1(\ell)\bigr)\le \Delta_2,
    \]
    and similarly
    \[
        \wabstract_2(\ell)-\wabstract_1(\ell)\le \Delta_1.
    \]
    Therefore
    \[
        \bigl|\wabstract_1(\ell)-\wabstract_2(\ell)\bigr|\le \Delta_1+\Delta_2.
    \]
    So if an endpoint of $a$ lies on the boundary of a collar about $\gamma$, then the two corresponding collar boundaries are separated by at most $\Delta_1+\Delta_2$, and the truncated length of $a$ changes by at most that amount.
    
    Now consider a cusp endpoint.  For each $i=1,2$, write $ h_i:=h_0(\wabstract_i)$ and define
    \[
        \varepsilon_i(\ell):=\wstd(\ell)-\wabstract_i(\ell)\in[0,\Delta_i].
    \]
    Then
    \[
        \cosh(\wabstract_i(\ell))
        =\cosh\bigl(\wstd(\ell)-\varepsilon_i(\ell)\bigr)
        \ge e^{-\varepsilon_i(\ell)}\cosh(\wstd(\ell))
        \ge e^{-\Delta_i}\cosh(\wstd(\ell)).
    \]
    Multiplying by \(\ell\) and letting \(\ell\to 0\), and using that for the standard collar
    \[
        \lim_{\ell\to 0}\ell\cosh(\wstd(\ell))=2,
    \]
    we obtain
    \[
    h_i\ge 2e^{-\Delta_i}.
    \]
    On the other hand, since \(\wabstract_i\le \wstd\), we also have \(h_i\le 2\). Thus
    \[
        \bigl|\log(h_1/h_2)\bigr|\le \Delta_1+\Delta_2.
    \]
    In the cusp model, the distance between the two horocycles \(N_{h_1}(c)\) and \(N_{h_2}(c)\) is exactly \(|\log(h_1/h_2)|\), so the change at a cusp endpoint is again at most \(\Delta_1+\Delta_2\).
    Therefore
    \[
    \ell^{\wabstract_2}_{X,\Gamma}(a)
    \le \ell^{\wabstract_1}_{X,\Gamma}(a)+2(\Delta_1+\Delta_2)
    \le L_A^1+2(\Delta_1+\Delta_2)
    = L_A^2.
    \]
    
    Since this holds for every \(a\in\mathcal A\), the decomposition \((\Gamma,\mathcal A)\) is \((L^2, L_A^2,\wabstract_2)\)-short.
\end{proof}

\section{Bounding shear parameters from above}

\subsection{Preliminary computations}

\begin{lemma}\label{lem:areaDisk}
    Let $\delta_1$ be the area of a region in an ideal triangle delimited by its incribed circle and its contact triangle, then \[ \delta_1 = \pi\left(\frac{2}{3} \cosh\left(2\rho\right)-1\right)-1\approx0.2768065 .\]
\end{lemma}

\begin{figure}[h]
    \centering
    \begin{overpic}[width=0.55\linewidth,keepaspectratio]{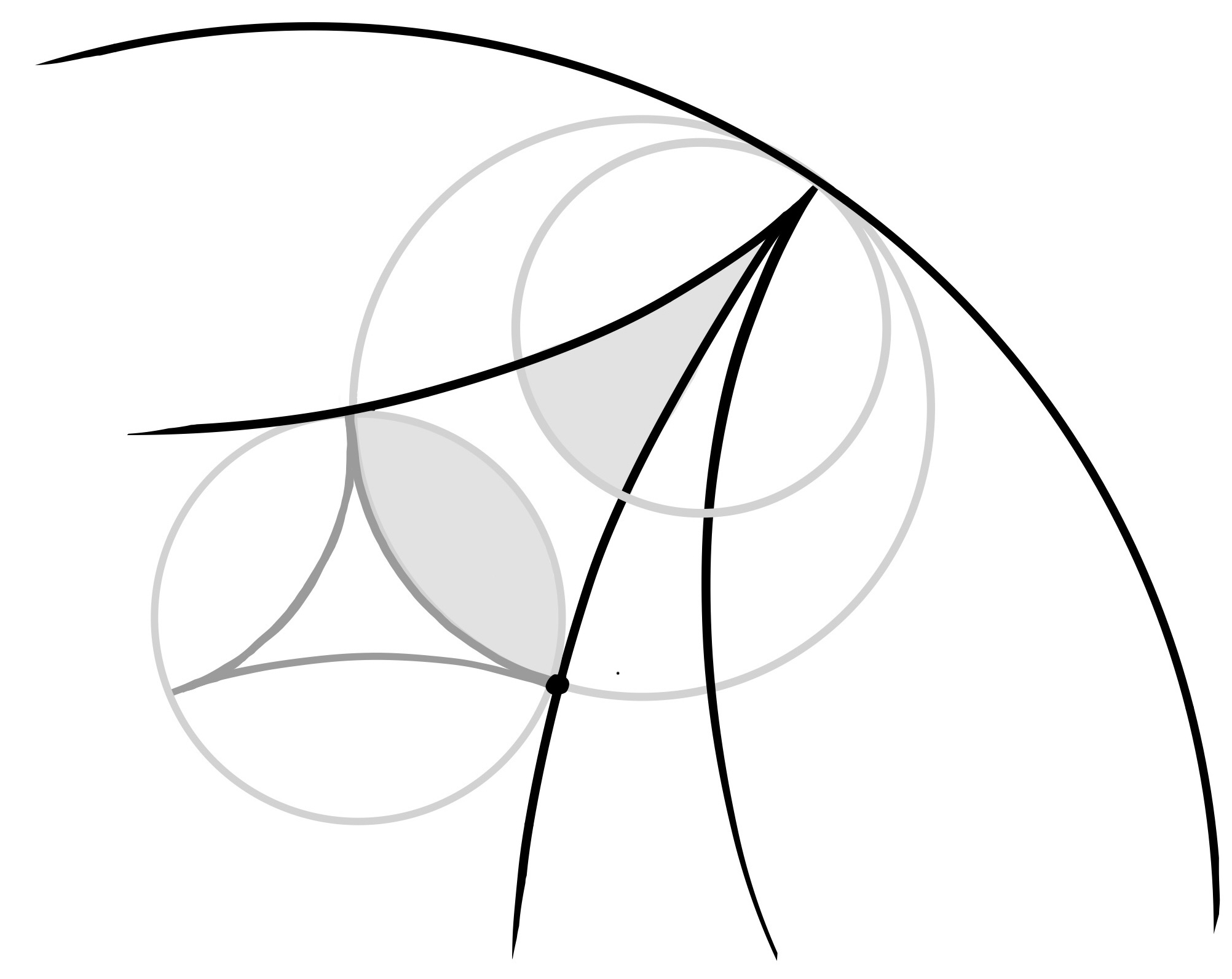}
    \put(36,34){{\color{Gray}$\mathcal{D}$}}
    \put(15,46){{$\Delta$}}
    \put(47,46){{\color{Gray}$R$}}
    \put(5,30){{\color{Gray}$C_{\Delta}$}}
    \put(29,21){{\color{Gray}$T_{\Delta}$}}
    \put(44.5,57){{\color{lightgray}$h_s$}}
    \put(46,20){$s$}
\end{overpic}\vspace*{-0.25cm}
    \caption{Computation of the area $\delta_1$ of the region $\mathcal{D}$ between the inscribed circle and the contact triangle in an ideal triangle $\Delta$.}
    \label{ComparingArea}
\end{figure}
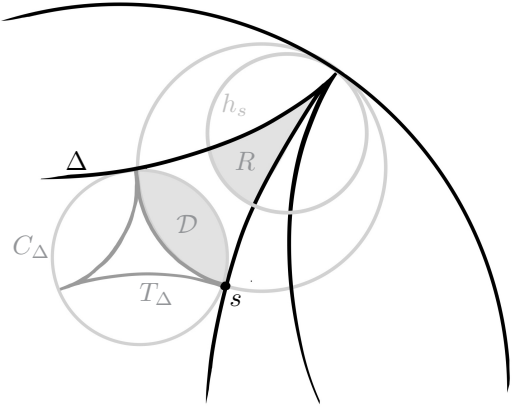

\begin{proof}
    Consider an ideal triangle $\Delta=(0,1,\infty)$ in $\Hh$. We compute the area, denoted by $\delta_1$, of the circular sector delimited by its inscribed circle $C_{\Delta}$ of radius $2\rho$ and the (non-geodesic) contact triangle $T_{\Delta}$. Let $R$ be a \emph{spike region} in $\Delta$ delimited by two edges of $\Delta$ and one side of its contact triangle. First observe that 
    \[\area(T_{\Delta})=\area(\Delta)-3\area(R),\] then by Gauss-Bonnet \[\area(\Delta)=\pi \quad \area(R)=\pi -\pi/2-\pi/2 + 1=1,\]
    and so we obtain
    \begin{align*}
        \delta_1&=\frac{1}{3}\left(\area(C_{\Delta})-\area(T_{\Delta})\right)=\frac{1}{3}\left(2\pi \left( \cosh\left(2\rho\right) - 1 \right)-(\pi-3)\right)\\ 
        &= \pi\left(\frac{2}{3} \cosh\left(2\rho\right)-1\right)-1\approx0.2768065.
    \end{align*}
\end{proof}

Recall that the injectivity radius of $X$ at a point $B$ is the radius of the largest open metric disk embedded at $B$; it is denoted by $r_B(X)$. The injectivity radius at $B$ satisfies:
\[
r_B(X) = \tfrac{1}{2}\,\ell(\mu_B),
\]
where $\mu_B$ denotes the shortest geodesic loop based at $B$, see for instance~\cite[Lemma~4.1.5]{Buser}. The result below is standard, but we include a proof for completeness.

\begin{lemma}\label{lem:InjRadiusLengthCusp}
    Let $X$ be a hyperbolic surface with at least one puncture, then a point with injectivity radius $r$ in a cusp neighborhood of $X$ lies on a horocycle of length $2\sinh(r)$.
\end{lemma}
\begin{proof}
    Let $B$ be a point in a cusp neighborhood of $X$ with injectivity radius $r$. Recall that there exists a geodesic loop of length $2r$ at $B$. We consider a lift of the cusp neighborhood in the upper half plan so that $B$ has two lifts $\widetilde{B}_1=(0,y)$ and $\widetilde{B}_2=(1,y)$. Then, \[\cosh(2r)=\cosh(d_\Hh(\widetilde{B}_1, \widetilde{B}_2))=1+\frac{(1-(-0))^2}{2y^2},\] and so $\frac{1}{y}=2\sinh(r)$ and $B$ lies on a horocycle of length $2\sinh(r)$.
\end{proof}

\subsection{Rough bound for ideal triangulations}
To illustrate the idea underlying the proof of the main result of this work, we first present a naive argument for the case of ideal triangulations.

Let $X \in \teich_{g,n}$ with $2g - 2 + n > 0$, and assume $n \geq 1$. Let $p_1, \ldots, p_n$ denote the punctures of $X$, where each puncture corresponds to a cusp $c_i$. Our objective is to construct an ideal triangulation of $X$ with bounded shear parameters. To do this, we consider the Voronoï cell decomposition of $X$ with respect to its set of punctures. We follow a construction similar to the one in the proof of Proposition~\ref{Hexagon-Decomposition-Existence}. This approach is already described in~\cite[Chapter 2 Section 3.i]{Harer1988}, \cite{BowditchEpstein1988}, \cite[Section 2.1]{Mondello2009}, and~\cite[Section 2.2]{Budd2025}. In this setting, the shortest geodesic joining a point $x \in X$ to a cusp $c_i$ is defined as the geodesic joining $x$ to $\partial N_{\epsilon}(c_i)$ that minimizes the truncated length for a sufficiently small $\epsilon > 0$. Applying the exact same procedure---constructing the dual graph of the spine and adding arcs if the spine is not trivalent---yields a well-defined ideal triangulation $\mathcal{T}$ of $X$.

Fix $\delta'_1 = 0.27 < \delta_1$. For each ideal arc $a \in \mathcal{T}$, let $a^t$ denote the subarc of $a$ lying outside the cusp neighborhoods $N_{\delta'_1}(c_i)$. Let $\Delta$ be an ideal triangle containing $a$ a boundary edge, and let $s$ be the shear point of $\Delta$ lying on $a$. 

\begin{claim}
    The shear point $s$ does not lie in $N_{\delta'_1}(c_i)$ for all $i=1\cdots, n$.
\end{claim}

\begin{proofclaim}
    Assume by contradiction that $s$ lies within the cusp neighborhood $N_{\delta'_1}(c_i)$. Recall that the shear point can be obtained by foliating a spike of $\Delta$ with horocycles. Hence $s$ must lie on a horocycle $h_s$ of length less than or equal to $\delta'_1$. On one hand, by Lemma~\ref{lem:areaDisk}, the area of the region $\mathcal{D} \subset \Delta$ bounded by its inscribed circle and its contact triangle is exactly $\delta_1$. On the other hand, the area of the cusp neighborhood cut off by $h_s$ satisfies
    \[\area(N_{h_s}(c_i)) \leq \area(N_{\delta'_1}(c_i))=\int_{1/\delta'_1}^{\infty}\int_{0}^{1} \frac{dxdy}{y^2}= \delta'_1.\] 
    By definition of the shear point $D \subset N_{h_s}(c_i)$ hence
    \[ \delta_1 = \area(D)\leq \area(N_{h_s}(c_i))\leq \delta'_1.\] which is a contradictio, thus $s$ does not lie in $N_{\delta'_1}(c_i)$. Therefore the shear along any ideal arc $a \in \mathcal{T}$ is bounded by the length of its truncated subarc $\ell_X(a^t)$.
\end{proofclaim}

Next, consider an embedded tubular neighborhood $U$ of width
\[r=\min\left(\frac{\sys(X)}{2}, \frac{\delta_1'}{2}\right)\] 
around $a^t$. By a standard hyperbolic area formula (see, for example~\cite[Lemma 3.6]{buser2018}), we have
\[\area(U)=2\left(\ell_X(a^t)\sinh\left(r\right)+\pi\left(\cosh\left(r\right)-1\right)\right).\] 
Since $\area(U)\leq \area(X)-n \cdot \delta_1'=2\pi(2g-2+n)-n\delta_1'$, it follows that:
\begin{align*}
    \left|S^{\mathcal{T}}_{a}(X)\right| &\leq \ell_X(a^t) \\
    &\leq \frac{2\pi(2g-2+n)-0.27\cdot n - 2\pi\left(\cosh\left(\min\left(\frac{\sys(X)}{2}, \frac{0.27}{2}\right)\right)-1\right)}{2\sinh\left(\min\left(\frac{\sys(X)}{2}, \frac{0.27}{2}\right)\right)}.
\end{align*}

Jiang established an upper bound that depends on the systole and on the genus times the Bers constant~\cite{ManmanJiang2021}. By considering not only ideal arcs ending in cusps but also geodesic arcs spiralling around simple closed curves, we improve our result and extend it to any hyperbolic surface: our bound depends logarithmically on the genus and the number of punctures and no longer on the systole. The proof follows the same strategy, first we use an area argument to locate regions where the shear points cannot occur. Then we estimate the shear along arcs by considering the subarcs that lie outside these regions.

\subsection{Bounding shear parameters along spiralling triangulations}

\begin{definition}
Let $\epsilon_T = \frac{3\rho}{2}$. Define the \emph{shear truncation function} by
\begin{equation*}
    \wtrunc \colon \R_{> 0} \to \R_{\geq 0}, \quad 
    \wtrunc(\ell) = \begin{cases}
        
        0 & \text{if } \ell \geq \epsilon_T, \\
        \arccosh\left(\frac{\epsilon_T}{\ell}\right) & \text{if } \ell < \epsilon_T.
    \end{cases}
\end{equation*}
The difference between the width of the standard collar $\wstd$ and the shear truncation is uniformly bounded from above by
\[ \Delta_{\wtrunc} < \wstd(\epsilon_T) < 2.28, \]
moreover $\wtrunc$ satisfies $h_0(\wtrunc) = \epsilon_T$.
\end{definition}

\begin{proposition}\label{lem:ShearFree}
Let $X \in \teich_{g,n}$ with $2g-2+n > 0$, and let $\mathcal{T}$ be a spiralling triangulation of $X$. Then the shear points of $\mathcal{T}$ lie in the $(\Gamma_{\wtrunc}, \wtrunc)$-\emph{truncation} of $X$, denoted by $X_{\Gamma_{\epsilon_T}}^{\wtrunc}$, where 
\[ 
    \Gamma_{\epsilon_T} \coloneq \left\{ \gamma \text{ is a simple closed curve on } X \mid \ell_X(\gamma) \leq \epsilon_T \right\}. 
\]
\end{proposition}

\begin{proof}
Let $s$ be a shear point of the ideal triangulation $\mathcal{T}$ on $X$. We show that $s$ lies within the truncated surface
\[ X^{w}_{\Gamma_{\epsilon_T}} \coloneqq X \setminus \left( \bigcup_{i=1}^n N_{\epsilon_T}(c_i) \cup \bigcup_{\gamma \in \Gamma_{\epsilon_T}} \mathcal{C}_{\wtrunc(\ell_X(\gamma))}(\gamma) \right). \]
To do so, we rely on the following geometric property, see Figure~\ref{BPointShear}:
\begin{center}
    ($\ast$) The point $s$ lies at a distance $\rho$ from a point $B$ whose injectivity radius satisfies  $r_B \geq \rho$.
\end{center}

\begin{figure}
    \centering
    \begin{overpic}[width=0.4\linewidth,keepaspectratio]{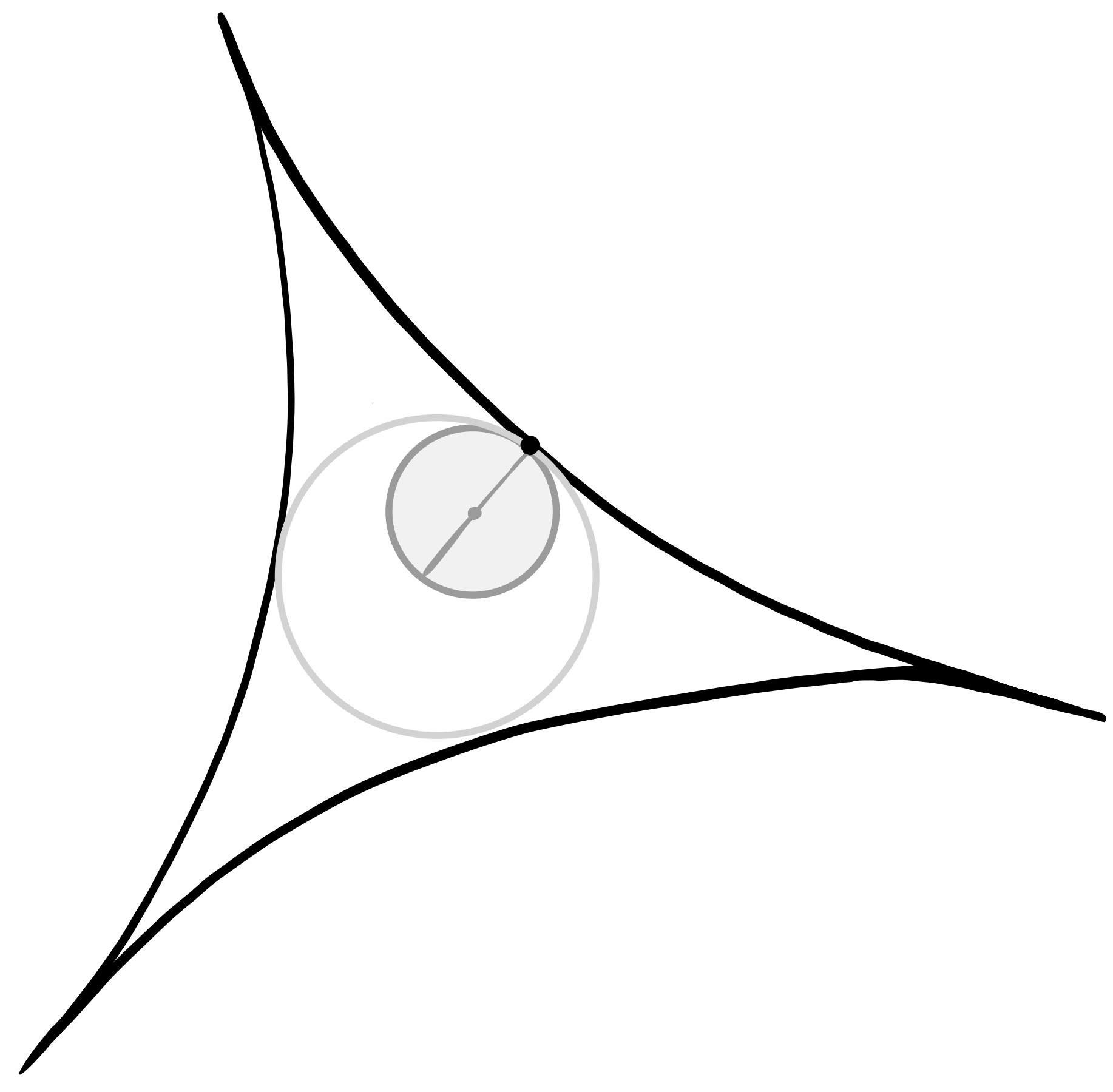}
        \put(49,59){$s$}
        \put(42.5,47){\color{Gray}$B$}
    \end{overpic}\vspace*{-0.25cm}
    \caption{Each shear point lies at a distance $\rho$ from a point $B$ with injectivity radius $r_B \geq \rho$.}
    \label{BPointShear}
\end{figure}

First, assume by contradiction that $s$ lies in a cusp neighborhood $N_{\epsilon_T}(c)$ for some cusp $c$ on $X$. On the one hand, $s$ lies on a horocycle of length at most $\epsilon_T$. Hence property~($\ast$) implies that $B$ must lie on a horocycle of length at most
\[ e^{\rho} \cdot \epsilon_T < 2\sinh(\rho). \]
On the other hand, by Lemma~\ref{lem:InjRadiusLengthCusp}, since the injectivity radius at $B$ satisfies $r_B \geq \rho$, $B$ must lie on a horocycle of length at least $2\sinh(r_B) \geq 2\sinh(\rho)$. This is a contradiction.

Finally, let $\gamma \in \Gamma_{\epsilon_T}$ and suppose for contradiction that $s$ lies in the collar $\mathcal{C}_{\wtrunc(\ell_X(\gamma))}(\gamma)$. Observe that for any length $\ell \leq \epsilon_T$, we have $\wtrunc(\ell) + \rho < \wstd(\ell)$. Consequently, the collars $\mathcal{C}_{\wtrunc(\ell_X(\gamma))}(\gamma)$ and $\mathcal{C}_{\wtrunc(\ell_X(\gamma)) + \rho}(\gamma)$ are contained within the standard collar $\mathcal{C}_{\wstd(\ell_X(\gamma))}(\gamma)$ (see Figure~\ref{Case1}).

\begin{figure}
    \centering
    \begin{overpic}[width=0.5\linewidth,keepaspectratio]{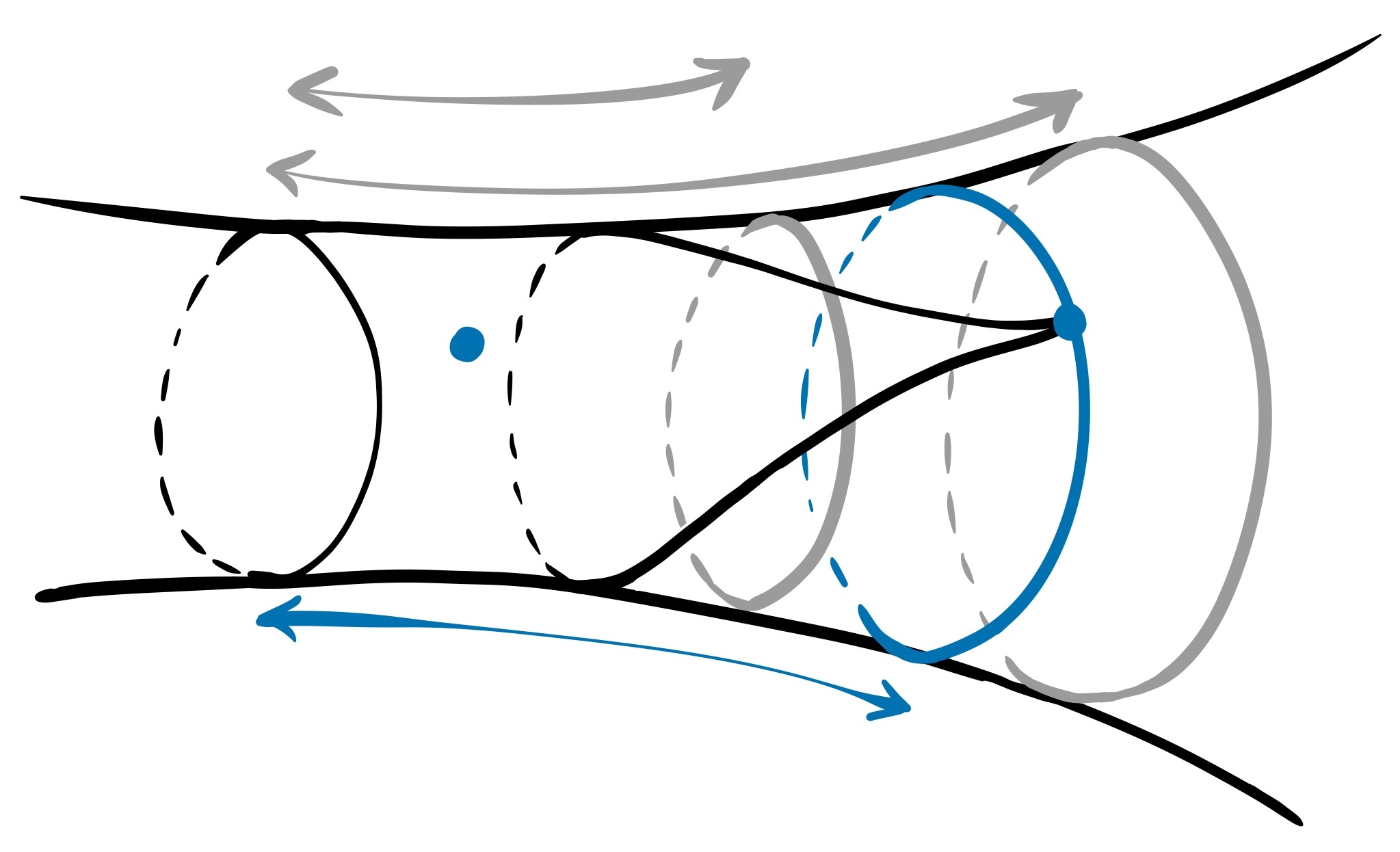}
        \put(29,36){\color{RoyalBlue}$s$}
        \put(78,37){\color{RoyalBlue}$B$}
        \put(22,30){\large$\gamma$}
        \put(59,56){\large\color{gray}$\wstd(\ell_X(\gamma))$}
        \put(24,58){\large\color{gray}$\wtrunc(\ell_X(\gamma))$}
        \put(37,10){\large\color{RoyalBlue}$d$}
    \end{overpic}\vspace*{-0.25cm}
    \caption{Position of the shear point relative to a collar around a short simple closed geodesic.}\label{Case1}
\end{figure}
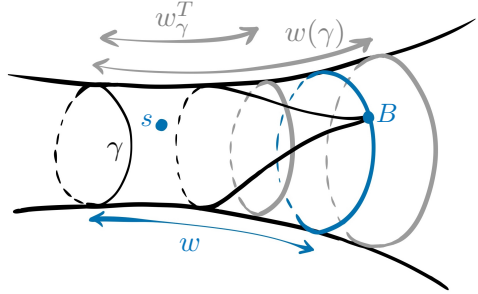

Let $d$ denote the distance from $B$ to $\gamma$. By construction $d$ satisfies
\[ d \leq \wtrunc(\ell_X(\gamma)) + \rho. \]
The point $B$ lies on the boundary curve of the collar $\mathcal{C}_d(\gamma)$, which has length $\ell_X(\gamma)\cosh(d)$. 
Since $r_B$ is half of the length of the shortest geodesic loop based at $B$, the injectivity radius at $B$ satisfies
\begin{align*}
    r_B &\leq \frac{\ell_X(\gamma)}{2}\cosh(d) < \frac{\ell_X(\gamma)}{2}\cosh\bigl(\wtrunc(\ell_X(\gamma)) + \rho\bigr) < \rho
\end{align*}
which contradicts property~($\ast$). The last inequality follows by studying the function on the left side. 

Consequently, the shear points must belong to the $(\Gamma_{\wtrunc}, \wtrunc)$-truncation of $X$, completing the proof.
\end{proof}

\begin{proposition}\label{ShortHexSHortTriangl}
      Let $X\in \teich_{g,n}$ with $2g-2+n>0$, let $L, L_A>0$ and let $(\Gamma, \As)$ be a $(L,L_A,\wtrunc)$-short hexagon decomposition of $X$. Then the shear parameter along any arc component $\lambda$ of $\mathcal{T}_{(\Gamma, \As)}$ is bounded by
      \[
        |S^{\mathcal{T}_{(\Gamma, \mathcal{A})}}_{\lambda}(X)|\leq 4L + 4\max(L_A,\epsilon_T)+ 3\epsilon_T + 8\rho.
      \]
      where $\rho=\frac{\log(3)}{4}$ and $\epsilon_T = \frac{3\rho}{2}$.
      %$L>\max(2R_{g,n}, 2D)$ where $D\coloneq\frac{R_{g,n}}{2} + \Delta_{\wabstract}$. Let $(L,\wabstract)$-short hexagon decomposition $(\Gamma, \As)$ on $X$.
\end{proposition}

\begin{proof}
    Let $(\Gamma, \As)$ be a $(L, L_A,\wtrunc)$-short hexagon decomposition of $X$ and consider the associated spiralling  triangulation 
    \[\mathcal{T}_{(\Gamma, \As)}= (\lambda_a)_{a\in \mathcal{A}}. \]
    We want to bound the shear parameters along the arcs in $\mathcal{T}_{(\Gamma, \As)}$. 

    The $(\Gamma,\wtrunc)$-truncation of $X$ is the truncated sub-surface $X^{\wtrunc}_{\Gamma}$ obtained by removing the cusp $\epsilon_T$-neighborhoods and the $\wtrunc$-collars around simple closed curves in $\Gamma$
    of length at most $\epsilon_T$:
    \[X^{\wtrunc}_{\Gamma}\coloneqq X \smallsetminus\left(\bigcup_{i=1}^n N_{\epsilon_T}(c_i) \cup \bigcup_{\gamma\in\Gamma \cap \Gamma_{\epsilon_T}} \mathcal{C}_{\wtrunc(\ell_X(\gamma))}(\gamma)\right),\]
    where 
    \[ \Gamma_{\epsilon_T} \coloneq \left\{ \gamma \text{ is a simple closed curve on } X \mid \ell_X(\gamma) \leq \epsilon_T \right\}. \]

    First assume that  $\lambda_a \in \mathcal{T}_{(\Gamma, \As)}$ is an ideal arc ending in two cusps and so \[\lambda_a=a.\] Proposition~\ref{lem:ShearFree} implies that the two shear points along $\lambda_a$ lie in the $(\Gamma,\wtrunc)$-truncation of $X$. Consequently, the shear parameter is bounded by the length of the truncated arc $a^{\wtrunc}_{\Gamma}$, and therefore
    \[
    \bigl|S^{\mathcal{T}_{(\Gamma,A)}}_{\lambda_a}(X)\bigr|
    \leq \ell^{\wtrunc}_{X,\Gamma}(a)
    \leq L_A.
    \]
    Hence it remains to consider the case where $a$ has at least one endpoint on a
    curve of $\Gamma$.

    Secondly, consider a hexagon $H$ in $(\Gamma, \mathcal{A})$ whose orthogeodesic arcs do not all end in two cusps. Denote by $T_H$ the corresponding ideal triangle in $\mathcal{T}_{(\Gamma, \mathcal{A})}$. Let
    \[ H^{\wtrunc} \coloneq H \cap X_{\Gamma}^{ \wtrunc}.\] 

    We need the two following claims to conclude the proof.

        \begin{claim}\label{claimA}
            Let $x\in H^{\wtrunc} \cap T_H$, then $x$ lies at distance at most $D_H=2L_A+2\max(L,\epsilon_T)+\epsilon_T$ to any of the three sides of $T_H$.
        \end{claim}

        \begin{proofclaim}    
            Let $\lambda_a, \lambda_b, \lambda_c$ be the three edges of $T_H$ where $a,b,c \in \As$. Let $x\in H^{\wtrunc} \cap T_H$. The radius of the inscribed circle in an ideal triangle is $2\rho$. If $x$ is the center of the inscribed circle of $T_H$, then 
            \[d(x,\lambda_a)=d(x,\lambda_b)=d(x,\lambda_c)= 2\rho<D, \]
            and the claim follows.
            Otherwise $x$ lies in of the spikes of $T_H$, that is a region delimited by two sides of $T_H$ and two radii of the inscribed circle connecting its center to a shear point. Without loss of generality, assume that the spike is bounded by $\lambda_a$ and $\lambda_b$. In an ideal triangle, every point is at distance at most $2\rho$ from its closest side, and a point in a spike is at distance at most $4\rho$ from the two sides bounding that spike. Hence
            \[
                d(x,\lambda_a)\leq 2\rho,
                \qquad
                d(x,\lambda_b)\leq 4\rho .
            \]
            It remains to bound $d(x,\lambda_c)$. 
            Since $(\Gamma, \As)$ is $(L,L_A, \wtrunc)$-short, every side of $H^{\wtrunc}$ has length\begin{itemize}
                \item[-] at most $L_A$ when it comes from an arc in $\mathcal{A}$;
                \item[-] at most $L$ when it comes from a curve in $\Gamma$;
                \item[-] at most $\epsilon_T$ when it comes from a boundary component of a cusp neighborhood or a collar around a curve in $\Gamma \cap \Gamma_{\epsilon_T}$.
            \end{itemize}  
            Hence the perimeter of $H^{\wtrunc}$ is at most $3L_A+3\max(L,\epsilon_T)$. In particular is it finite so $H^{\wtrunc}$ does not contain the ideal vertex of $T_H$ shared by $\lambda_a$ and $\lambda_b$. Consequently, there exists a connected subarc $e \subset \partial H^{\wtrunc}$ with endpoint on $\lambda_a$ and $\lambda_b$. The subarc $e$ separates the spike in $T_H$ delimited by $\lambda_a$, $\lambda_b$ into two regions: one containing the ideal vertex while the other contains a subset $H^{\wtrunc}$ and the point $x$%, see Figure~\ref{hatHbounded}. 
            . Let $\lambda_x$ be the unique geodesic ray starting at $x$ and asymptotic to the ideal vertex shared by $\lambda_a$ and $\lambda_b$. Because $e$ separates $x$ from that ideal vertex, $\lambda_x$ must intersect $e$ at some point $y \in e$. The distance to $\lambda_c$ strictly increases as one moves along $\lambda_x$ from $x$, deeper into the spike, towards the ideal vertex. It follows that
            \[
                d(x,\lambda_c)\leq d(y,\lambda_c).
            \] 
            \begin{itemize}
                \item \textbf{If $\lambda_c \cap H^{\wtrunc} \neq \emptyset$:} To travel from $y\in \partial H^{\wtrunc} \cap T_H$ to $\lambda_c$ we travel along at most four sides of $ \partial H^{\wtrunc}$ thus 
                \[ d(y,\lambda_c) \leq 2L_A+2\max(L,\epsilon_T).\]
                \item \textbf{If $\lambda_c \cap H^{\wtrunc}=\emptyset$:} Then the corresponding arc $c$ does not end in two cusps and $\lambda_c$ must spiral around some curve $\gamma \in \Gamma$. %see Figure~\ref{Figure30}
                \begin{itemize}
                    \item Assume $\gamma$ has length at least $\epsilon_T$. Since $\lambda_c$ spirals around $\gamma$ it intersects infinitely many times $c$. Moreover $H$ is not truncated at $\gamma$ thus $\lambda_c$ must intersect $H^{\wtrunc}$ which is a contradiction. See Figure~\ref{hatHbounded}.
                    \item Assume that $\gamma$ has length at most $\epsilon_T$. Let $\mu$ be the subarc of the boundary curve $\partial \mathcal{C}_{\wtrunc}(\gamma)$ contained in $\partial H^{\wtrunc}$. We travel from $y$ to $\mu$ along $\partial H^{\wtrunc}$ via at most four sides and then from $\mu$ to $\lambda_c$ along $\partial \mathcal{C}_{\wtrunc}(\gamma)$. Since  $\partial \mathcal{C}_{\wtrunc}(\gamma)$ has length $\ell_X(\gamma)\cosh\left(\wtrunc(\ell_X(\gamma))\right)=\epsilon_T$ we obtain:
                    \[d(y,\lambda_c) \leq 2L_A+2\max(L,\epsilon_T) + \epsilon_T. \]
                \end{itemize}
                
            \end{itemize}

            \begin{figure}
                \centering
                \begin{overpic}[width=1.0\linewidth,keepaspectratio]{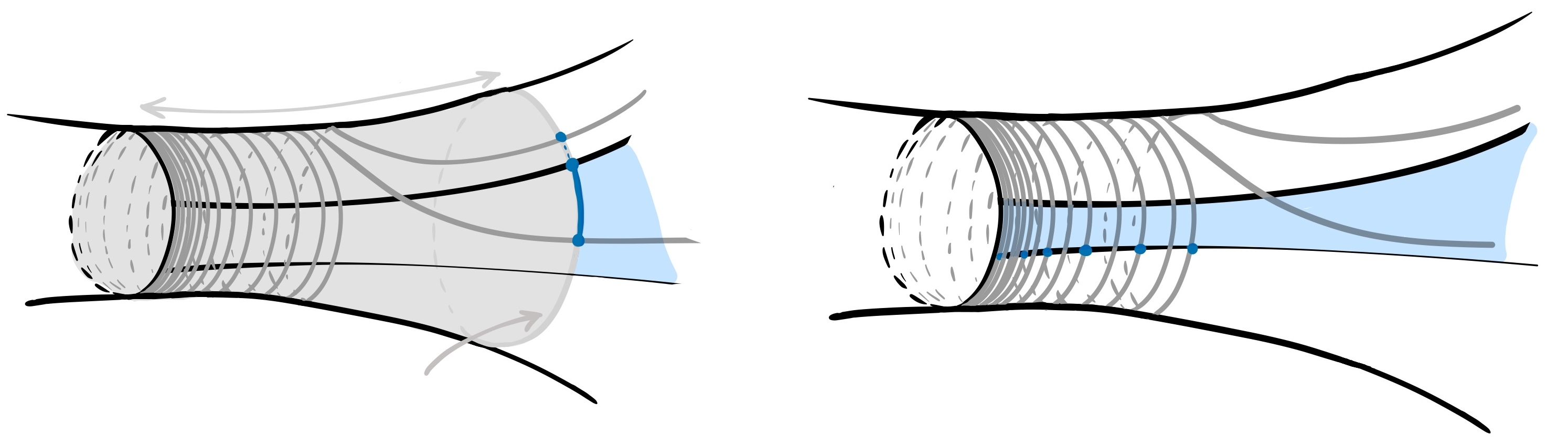}
                \put(41.5,22){\color{Gray}$\lambda_c$}
                \put(17,4){\color{lightgray}$\partial \mathcal{C}_{\wtrunc}(\gamma)$}
                \put(15,23.5){\color{lightgray}$\wtrunc\left(\ell_X(\gamma)\right)$}
                \put(43.5,10){$c$}
                \put(2.5,16){$\gamma$}
                \put(37.5,15){\color{RoyalBlue}$\mu$}
                \put(42,15){\color{MarieBlue}$H^{\wtrunc}$}

                \put(96,22){\color{Gray}$\lambda_c$}
                \put(98.5,11){$c$}
                \put(94,15){\color{MarieBlue}$H^{\wtrunc}$}
                \put(54.5,16){$\gamma$}
                \end{overpic}\vspace*{-0.25cm}
                \caption{Intersection between $\lambda_c$ and a truncated hexagon $H^{\wtrunc}$. On the left $\lambda_c \cap H^{\wtrunc} = \emptyset$ whereas on the right $\lambda_c \cap H^{\wtrunc} \neq \emptyset$.}
                \label{hatHbounded}
            \end{figure}

        \end{proofclaim}

        \begin{claim}\label{claimB}
            Let $x\in T_H$ such that $x$ is at distance at most $D$ to any of the three sides of $T_H$. Let $x_0$ be the center of the inscribed circle in $T_H$, then \[ d(x_0, x)\leq D + 2\rho. \]
        \end{claim}

        \begin{proofclaim}
            If $x$ is contained in the inscribed circle, we are done. Else, $x$ lies in one spike of $T_H$. Let $\lambda_c$ be the side opposite to this spike. The geodesic path from $x$ to $\lambda_c$, of length $\leq D$,  intersects the inscribed circle at some point $y$, which concludes the claim.
        \end{proofclaim}

        Let $H_1$ and $H_2$ be two hexagons sharing a common edge $a\in (\Gamma, \mathcal{A})$ that does not end in two cusps. Let $T_{H_1}$ and $T_{H_2}$ be the corresponding ideal triangles in $\mathcal{T}_{(\Gamma, \mathcal{A})}$ sharing the edge $\lambda_a$. Let $s_{\lambda_a}^1$ and $s_{\lambda_a}^2$ be the shear points associated with $T_{H_1}$ and $T_{H_2}$ along $\lambda_a$, and let $x_0^1$ and  $x_0^2$ be the respective inscribed circle centers of $T_{H_1}$ and $T_{H_2}$x. 

        Because $\lambda_a$ spirals infinitely around at least one ending curve of $a$,  the geodesic arcs $\lambda_a$ and $a$ must intersect. We distinguish two cases based on where this intersection occurs relative to the truncated regions.

        \noindent\textbf{Case 1: There exists an intersection point $x_{a,\lambda_a} \in a \cap \lambda_a$ that lies within $H_1^{\wtrunc} \cup H_2^{\wtrunc}$.}
        
        By the triangle inequality, the shear parameter satisfies:
        \[ 
            \lvert S_{\lambda_a}^{\mathcal{T}_{(\Gamma,\mathcal{A})}}(X)\rvert = d(s^1_{\lambda_a}, s^2_{\lambda_a}) \le d(s^1_{\lambda_a}, x_0^1) + d(x_0^1, x_{a,\lambda_a}) + d(x_{a,\lambda_a}, x_0^2) + d(x_0^2, s^2_{\lambda_a}). 
        \]
        By the definition of shear points, $d(x_0^i, s^i_{\lambda_a}) = 2\rho$. Because $x_{a,\lambda_a}$ belongs to both $\hat{H}_i$ and $T_{H_i}$, Claims 3.6 and 3.7 guarantee that $d(x_{a,\lambda_a}, x_0^i) \le D_H + 2\rho$. Therefore:
        \begin{align*}
            |S_{\lambda_a}^{\mathcal{T}_{(\Gamma,\mathcal{A})}}(X)| &\le 2\rho + (D_H + 2\rho) + (D_H + 2\rho) + 2\rho = 2D_H + 8\rho\\
            &= 4L + 4\max(L_A, \epsilon_T) +2\epsilon_T + 8\rho.
        \end{align*}
        
        \noindent\textbf{Case 2: No intersection point between $a$ and $\lambda_a$ lies within $H_1^{\wtrunc} \cup H_2^{\wtrunc}$.}
        
        This implies that spiralling intersection occurs entirely outside the truncated surface, that is the region removed from shear truncation is not empty. Consequently, $a$ must end on a boundary curve $\eta$, where $\eta = \partial N_{\epsilon_T}(c)$ or $\eta = \partial C_{\wtrunc}(\gamma)$ for some $\gamma$ of length $\le \epsilon_T$. Both $\lambda_a$ and $a$ must cross $\eta$. Let $x_1 = \lambda_a \cap \eta$ and $x_2 = a \cap \eta$ be these intersection points. Since both points lie on $\eta$, the distance between them is bounded by the length of $\eta$:
        \[ d(x_1, x_2) \le \ell_X(\eta) \le \epsilon_T. \]
        Applying the triangle inequality yields:
        \[ |S_{\lambda_a}^{\mathcal{T}_{(\Gamma,\mathcal{A})}}(X)| \leq d(s^1_{\lambda_a}, x_0^1) + d(x_0^1, x_1) + d(x_1, x_2) + d(x_2, x_0^2) + d(x_0^2, s^2_{\lambda_a}). \]
        Substituting our established bounds from Claims 3.6 and 3.7 gives:
        \[ |S_{\lambda_a}^{\mathcal{T}_{(\Gamma,\mathcal{A})}}(X)| \leq 2\rho + (D_H + 2\rho) + \epsilon_T + (D_H + 2\rho) + 2\rho = 2D_H + \epsilon_T + 8\rho. \]
        Substituting $D_H = 2L +2\max(L_A,\epsilon_T)+ \epsilon_T$, we obtain:
        \[ |S_{\lambda_a}^{\mathcal{T}_{(\Gamma,\mathcal{A})}}(X)| \leq 4L + 4\max(L_A,\epsilon_T)+ 3\epsilon_T + 8\rho. \]
        This finishes the proof.
\end{proof}

Corollary~\ref{corExistSdtShort} together with Proposition~\ref{ShortHexSHortTriangl} above  implies:
\begin{theorem}\label{thm:main}
       Let $X\in \teich_{g,n}$ with $2g-2+n>0$, then there exists a spiralling triangulation $\mathcal{T}_{(\Gamma, \As)}$ such that the shear parameter along any arc component $\lambda$ of $\mathcal{T}_{(\Gamma, \As)}$ is bounded by
      \[
        |S^{\mathcal{T}_{(\Gamma, \mathcal{A})}}_{\lambda}(X)|\leq 12\log(8\pi(2g-2+n))+25.
      \]
\end{theorem}

By Propositions~\ref{ComparingLshort} and~\ref{ShortHexSHortTriangl} we obtain:
\begin{theorem}
      %$L>\max(2R_{g,n}, 2D)$ where $D\coloneq\frac{R_{g,n}}{2} + \Delta_{\wabstract}$. Let $(L,\wabstract)$-short hexagon decomposition $(\Gamma, \As)$ on $X$.
       Let $X\in \teich_{g,n}$ with $2g-2+n>0$, let $\wabstract:\R_{>0}\rightarrow \R_{\geq 0}$ be a truncation function and $(\Gamma, \As)$ be a $(L,L_A,\wabstract)$-short hexagon decomposition of $X$. Then the shear parameter along any arc componxent $\lambda$ of $\mathcal{T}_{(\Gamma, \As)}$ is bounded by
      \[
        |S^{\mathcal{T}_{(\Gamma, \mathcal{A})}}_{\lambda}(X)|\leq 4L + 4\max(L_A+ 2(\Delta_{\wabstract} +\Delta_{\wtrunc}), \epsilon_T) + 3\epsilon_T +8\rho.
      \]
      where $\rho=\frac{\log(3)}{4}$ and $\epsilon_T = \frac{3\rho}{2}$.
\end{theorem}

\bibliographystyle{alpha}
\bibliography{bib.bib}
\end{document}

%% file: commands.tex
\theoremstyle{plain}
\newtheorem{theorem}{Theorem}[section]
\newtheorem{proposition}[theorem]{Proposition}
\newtheorem{lemma}[theorem]{Lemma}
\newtheorem*{lemma*}{Lemma}
\newtheorem{cor}[theorem]{Corollary}
\theoremstyle{definition}

\newtheorem{claim}[theorem]{Claim}
\newtheorem{definition}[theorem]{Definition}
\newtheorem{remark}[theorem]{Remark}

\newtheorem*{theorem**}{Theorem\theoremnum}
\newenvironment{theorem*}[1][]{%
  \edef\theoremnum{\if\relax\detokenize{#1}\relax\else~#1\fi}% Store theorem number
  \begin{theorem**}
}{%
  \end{theorem**}
}  

\newtheorem*{proposition**}{Proposition\theoremnum}
\newenvironment{proposition*}[1][]{%
  \edef\theoremnum{\if\relax\detokenize{#1}\relax\else~#1\fi}% Store theorem number
  \begin{proposition**}
}{%
  \end{proposition**}
}  

\newtheorem*{cor**}{Corollary\theoremnum}
\newenvironment{cor*}[1][]{%
  \edef\theoremnum{\if\relax\detokenize{#1}\relax\else~#1\fi}% Store theorem number
  \begin{cor**}
}{%
  \end{cor**}
}  

\newtheorem*{BersThm*}{Bers constant Theorem}

\newcommand{\addQEDstyle}[2]{\AtBeginEnvironment{#1}{\pushQED{\qed}\renewcommand{\qedsymbol}{#2}}\AtEndEnvironment{#1}{\popQED}}

\addQEDstyle{example}{$\triangle$}

\newenvironment{proofclaim}{%
  \begin{proof}[Proof of the claim]
  \renewcommand{\qedsymbol}{$\diamond$}%
}{%
  \end{proof}
}

\DeclareMathOperator{\std}{std}

\DeclareMathOperator{\spine}{Sp}

\DeclareMathOperator{\arccosh}{arccosh}
\DeclareMathOperator{\arcsinh}{arcsinh}

\DeclareMathOperator{\area}{area}

\DeclareMathOperator{\sys}{sys}
\DeclareMathOperator{\arcch}{arccosh}

\DeclareMathOperator{\teich}{Teich}

\newcommand{\R}{\mathbb{R}}
\newcommand{\Hh}{\mathbb{H}}

\newcommand{\As}{\mathcal{A}}

\newcommand{\wabstract}{w}
\newcommand{\wstd}{w^{\std}}
\newcommand{\wtrunc}{w^{T}}

\newcommand\blfootnote[1]{%
  \begingroup
  \renewcommand\thefootnote{}\footnote{#1}%
  \addtocounter{footnote}{-1}%
  \endgroup
}

\definecolor{MarieBlue}{HTML}{5AA6E4}